\documentclass{amsart}
\pdfoutput=1

\usepackage{microtype}
\usepackage[mathscr]{eucal}
\usepackage{mathtools}
\usepackage{amsthm}
\usepackage{enumitem}

\usepackage{hyperref}
\usepackage{mleftright}

\DeclareMathOperator{\inte}{int}
\DeclareMathOperator{\id}{id}

\theoremstyle{plain}
\newtheorem{theorem}{Theorem}[section]

\newtheorem{lemma}[theorem]{Lemma}
\newtheorem{corollary}[theorem]{Corollary}
\theoremstyle{remark}
\newtheorem{remark}[theorem]{Remark}
\theoremstyle{definition}

\begin{document}
	
\title[A mechanical characterization of CMC surfaces]{A mechanical characterization\\ of constant mean curvature surfaces}

\author{Matteo Raffaelli}
\address{School of Mathematics, Georgia Institute of Technology, Atlanta, Georgia 30332}
\email{raffaelli@math.gatech.edu}
\date{July 29, 2026}

\begin{abstract} 
The \textit{speed} of a ball rolling without skidding or spinning on a surface $S$ is the length of the velocity of its center. We show that if the speed depends only on $p\in S$, then $S$ has constant mean curvature; and, conversely, that if the mean curvature of $S$ is constant and equal to $H> 0$, then either $S$ is a sphere or the ball of radius $1/H$ rolls on $S$ with direction-independent speed. It follows that the only surfaces where the speed is constant are subsets of planes, circular cylinders, and spheres.
\end{abstract}
\maketitle

\section{Introduction}

In an interesting paper~\cite{pamfilos1986}, Pamfilos proved that a compact orientable surface in $\mathbb{R}^3$ is a sphere if and only if, given a sufficiently small ball rolling freely without skidding on the surface, the ball’s center moves with constant speed for any initial condition; see also \cite{raffaelli2025}.

This result links geometry and dynamics. Its proof is based on the d'Alembert--Lagrange principle, which the author uses to derive the equations of motion. It should be noted that the ball is allowed to spin while rolling.

Inspired by Pamfilos' paper, here we obtain similar links between the geometry of a connected orientable surface $S$ and the \emph{kinematics} of rolling motions, using surface theory alone. Following Nomizu's viewpoint~\cite{nomizu1978}, a \textit{rolling (without skidding or spinning)} of one surface on another is described extrinsically by a one-parameter family of rigid motions that keeps the moving surface tangent to $S$ along a prescribed contact curve $\gamma$. The no-skid condition means that the instantaneous relative velocity vanishes at the point of contact, while the no-spin condition means that the instantaneous rotation does not include any component about the common normal. Thus, in the absence of equations of motion, to define a kinematic rolling one must prescribe the entire locus of contact $\gamma$ between the two surfaces, rather than merely an initial condition. 

To state our main result, fix $p\in S$, let $B_{r\neq0}$ be the ball of radius $\lvert r\rvert$ centered at $p + r N_p$, where $N$ is the surface unit normal, and let $\gamma$ be a unit-speed curve in $S$ such that $\gamma(0) = p$. Under a mild condition on $\gamma$, there exists a unique rolling (without skidding or spinning) of $B_r$ on $S$ such that $\gamma(t)$ is a point of contact with $B_r$; see Theorem~\ref{TH:rolling}. The \textit{speed} of $B_r$ while rolling on $S$ is the length of the velocity of its center. We say that $B_r$ rolls \textit{isotropically from $p$} if its initial speed while rolling along $\gamma$ does not depend on the choice of direction $v\coloneqq \gamma'(0) \in T_pS$.

\begin{theorem}\label{TH:main}
If for some $r$ the initial speed of $B_{r}$ is the same for three pairwise nonparallel directions $v \in T_pS$, then either $p$ is an umbilic point of  $S$, or the mean curvature of $S$ with respect to $N$ equals $1/r$ at $p$. Conversely, if $S$ has mean curvature $h \neq 0$ and is umbilic (resp., nonumbilic) at $p$, then $B_{r}$ rolls isotropically from $p$ for all $r \neq 1/h$ (resp., only for $r = 1/h$).
\end{theorem}

\begin{remark}
If $p$ is umbilic and the mean curvature is $h\neq0$, then there is no rolling of $B_{1/h}$ on $S$.  Indeed, in that case the normal curvature and geodesic torsion of $\gamma$ at $\gamma(0)$ coincide with the corresponding quantities on the ball, which violates the existence condition in Theorem~\ref{TH:rolling}. 
\end{remark}

We emphasize that, although we use the term ``ball''  and the symbol $B_r$, it is technically the sphere bounding $B_r$ that rolls on $S$; see section~\ref{sec:rolling}. We further note that, in our setting, intersections between $B_r$ and $S$ are allowed.

As a consequence of the above theorem, we obtain the following mechanical characterizations of nonminimal constant mean curvature (CMC) surfaces and spheres. 

\begin{corollary}\label{COR:main}
Suppose that $S$ is not totally umbilic. Then $S$ has nonzero constant mean curvature if and only if for some $r$ the ball $B_r$ rolls isotropically from every nonumbilic point of $S$.
\end{corollary}
\begin{proof}
Let $U$ denote the set of umbilic points of $S$, and suppose that $B_r$ rolls isotropically from every point of $S\setminus U$. Then the mean curvature $H$ equals $1/r$ on $S\setminus U$. By continuity, $H=1/r$ on the closure $\overline{S\setminus U}$. Now let $W$ be any connected component of $\inte(U)$. Since every point of $W$ is umbilic, $W$ lies in a sphere or a plane, so $H$ is constant on $W$. Because $\partial W\subset \overline{S\setminus U}$, continuity implies $H=1/r$ on $W$ as well, completing the proof.
\end{proof}

\begin{remark}
On a connected CMC surface, umbilic points are isolated unless the surface is totally umbilic \cite[p.~239]{woodward2019}.
\end{remark}

\begin{corollary}
Suppose that $S$ is compact, and choose $r$ such that $1/\lvert r\rvert > \max_{ p\in S, i=1,2 } \lvert\kappa_i(p)\rvert$, where $\kappa_1,\kappa_2$ are the principal curvatures. Then $S$ is a sphere if and only if $B_r$ rolls isotropically from every point of $S$. 
\end{corollary}

\begin{remark}
The condition on $r$ in the last corollary ensures that the normal curvature of any curve on (the boundary of) $B_r$ never coincides with that of $\gamma$, and that $(\kappa_1(p)+\kappa_2(p))r\neq 2$ for all $p\in S$.
\end{remark}

In the case of spheres, not only is the speed of $B_r$ constant with respect to the direction $v \in T_pS$, but also with respect to $p$. Applying Theorem~\ref{TH:main}, we can identify all surfaces with this property.

\begin{corollary}\label{COR:ConstSpeed}
The only connected orientable surfaces where for some $r$ the speed of the ball $B_r$ is constant are (subsets of) planes, circular cylinders, and spheres. In particular, $r=2R$ for a cylinder of radius $R$ oriented inward.
\end{corollary}

Theorem~\ref{TH:main} (and Corollary~\ref{COR:ConstSpeed}) will be proved in section~\ref{sec:proof} on the basis of results of \cite{nomizu1978, raffaelli2018}, which we review in section~\ref{sec:rolling}. More explicitly, the proof relies on the fact that, as $B_r$ rolls on $S$ along $\gamma$, the angular velocity vector is completely determined by the Darboux frames of $\gamma$ and its ``anti-development'' $\widetilde\gamma$ on $B_r$ (the locus of contact on $B_r$, uniquely defined by the geodesic curvature of $\gamma$). This fact allows us to express the speed of $B_r$ solely in terms of $r$ and the Darboux curvatures of $\gamma$, and reduces the proof to a computation involving the principal curvatures of $S$.

It is worth noting that Theorem~\ref{TH:main} can be generalized by allowing the rolling surface $\widetilde{S}$ to be different from a ball, as long as we interpret the speed of $\widetilde S$ as the length of its angular velocity vector, and we position $\widetilde{S}$ so that $v\in T_pS$ is a fixed principal direction $\widetilde{e}_i$ of $\widetilde S$. Then a minor modification of the proof yields that $\widetilde S$ rolls isotropically from $p$ if and only if either $p$ is umbilic or the mean curvature of $S$ equals $\widetilde{\kappa}_i$, where $\widetilde{\kappa}_i$ is the principal curvature corresponding to $\widetilde{e}_i$. This more general statement implies that a developable surface can roll isotropically on a nonplanar minimal surface, while we already know that a ball cannot.

We conclude this introduction with some remarks on the literature. The problem of rolling two surfaces, or more generally two submanifolds, one on the other has a rich history. First considered by Chaplygin in 1897~\cite{chaplygin1897}, it was later studied by Nomizu~\cite{nomizu1978}, Bryant--Hsu~\cite[section~4]{bryant1993}, Sharpe~\cite[Appendix~B]{sharpe1997}, and Agrachev--Sachkov~\cite[Chapter~24]{agrachev2004}, among others. Of particular interest nowadays is the intrinsic viewpoint, first presented in \cite{godoy2012}. For a detailed historical perspective, discussing both classical and modern results, we refer the reader to the excellent survey \cite{chitour2014}.

\section{Kinematic rolling}\label{sec:rolling}
In this section we discuss some preliminaries. In the first part we follow rather closely \cite{nomizu1978}. Relevant textbook references are \cite[section~7.9]{woodward2019} and \cite[section~3.5]{robbin2022}.

Let $t \in [0,\ell]$. Consider a smooth one-parameter family $f_t$ of direct isometries $x\mapsto A_tx + b_t$ of $\mathbb{R}^3$ such that $f_0 = \id$, and let $X_t$ be the time-dependent vector field $x \mapsto f_t'(f_t^{-1}(x))$. We call $X_t$ the \textit{instantaneous motion} at instant $t$. A computation shows that 
\begin{equation*}
f_t'(f_t^{-1}(x)) = Q_t x +v_t,
\end{equation*}
where $Q_t\coloneqq A_t'A_T^T$ and $v_t \coloneqq b'_t-Q_t b_t$. In particular, the instantaneous motion is called
\begin{enumerate}
\item an \textit{instantaneous standstill} if $Q_t=0$ and $v_t=0$; 
\item an \textit{instantaneous translation} if $Q_t=0$ and $v_t\neq 0$; 
\item an \textit{instantaneous rotation} if $Q_t\neq 0$ and there is a point $x_0$, called \textit{center}, such that $f_t'(f_t^{-1}(x_0))=0$.
\end{enumerate}
Note that if $X_t$ is an instantaneous rotation, then there exists a unique vector $\omega_t$, the \textit{angular velocity} vector, such that 
\begin{equation*}
Q_t x =\omega_t\times x \quad \forall x\in \mathbb{R}^3.
\end{equation*}

We are now ready to define rolling. Let $S$ and $\widetilde{S}$ be connected $\mathcal{C}^2$ surfaces tangent to each other at a point $p\in S$, and let $\gamma\colon [0,\ell]\to S$ be a unit-speed curve such that $\gamma(0) = p$. A \textit{motion} of $\widetilde{S}$ on $S$ along $\gamma$ is a map $f_t$ with the property that $f_t(\widetilde S)$ is tangent to $S$ at $\gamma(t)$. A motion of $\widetilde{S}$ is said to be \textit{rotational} if $X_t$ is an instantaneous rotation with $\gamma(t)$ as center. A rotational motion is a \textit{rolling (without skidding or spinning)} if the angular velocity vector is tangent to $S$ at $\gamma(t)$.

To construct a rolling of $\widetilde S$ on $S$ along $\gamma$, we proceed essentially as in \cite{raffaelli2018}. Let $N$ be a unit vector field normal to $S$ along $\gamma$, and let $D_t$ be the matrix whose columns are the elements of the Darboux frame $(\gamma'(t), N(t)\times \gamma'(t), N(t))$ of $\gamma$. For any unit-speed curve $\widetilde\gamma \colon [0,\ell]\to \widetilde S$ such that $\widetilde\gamma(0) = p$ and $\widetilde\gamma'(0) =\gamma'(0)$, let $\widetilde{D}_t$ be defined analogously by the normal $\widetilde N$ that coincides with $N$ at $p$, so that $D_0=\widetilde{D}_0$. Then, setting
\begin{align*}
A_t &\coloneqq D_t\widetilde{D}_t^T,\\
b_t &\coloneqq \gamma(t)-A_t \widetilde{\gamma}(t),
\end{align*}
the family $f_t$ becomes a motion of $\widetilde{S}$ on $S$ along $\gamma$~\cite[Prop.~2.2]{raffaelli2018}. Since 
\begin{equation*}
f_t'(f_t^{-1}(x)) = Q_t(x-\gamma(t)),
\end{equation*}
we observe that $f_t$ is rotational when $Q_t\neq 0$. In particular, we may compute that $Q_t =D_t' D_t^T + D_t \widetilde{D}_t'^T \widetilde{D}_t D_t^T$. Next let
\begin{equation*}
\varLambda_t \coloneqq
\begin{bmatrix}
0 & \kappa_g(t) & \kappa_n(t)\\
-\kappa_g(t) & 0 & \tau_g(t) \\
-\kappa_n(t) & -\tau_g(t) & 0
\end{bmatrix},
\end{equation*}
where $\kappa_g\coloneqq \langle \gamma'', N\times \gamma' \rangle$, $\kappa_n \coloneqq \langle \gamma'', N\rangle$, and $\tau_g \coloneqq -\langle N', N\times \gamma'\rangle$ are the geodesic curvature, the normal curvature, and the geodesic torsion of $\gamma$. Since $D_t' = D_t\varLambda_t^T$ and $\varLambda_t$ is skew-symmetric, we obtain $Q_t = D_t(\widetilde{\varLambda}_t-\varLambda_t)D_t^T$. Hence
\begin{equation*}
\omega_t = (\tau_g(t)-\widetilde{\tau}_g(t))\gamma'(t)- (\kappa_n(t)-\widetilde{\kappa}_n(t))N(t)\times \gamma'(t) + (\kappa_g(t)-\widetilde{\kappa}_g(t))N(t).
\end{equation*} 

The following theorem concludes our discussion.
\begin{theorem}[\cite{nomizu1978,raffaelli2018}]\label{TH:rolling}
Suppose that $\widetilde S$ is complete, and let $\widetilde\gamma^\ast$ be the unique curve $\widetilde \gamma$ whose geodesic curvature coincides with $\kappa_g$. Then there exists a (unique) rolling of $\widetilde S$ on $S$ along $\gamma$ if and only if $\kappa_n$ and $\tau_g$ are never simultaneously equal to $\widetilde{\kappa}_n^\ast$ and $\widetilde{\tau}_g^\ast$. Moreover, if $\widetilde S$ rolls on $S$ along $\gamma$, then the angular velocity has components $(\tau_g-\widetilde{\tau}_g^\ast, \widetilde{\kappa}_n^\ast-\kappa_n, 0)$ in the Darboux frame of $\gamma$.
\end{theorem}

\begin{remark}
In the introduction we referred to the curve $\widetilde\gamma^\ast$ as the \emph{anti-development} of $\gamma$. When $\widetilde S$ is complete, standard ODE theory guarantees that $\widetilde\gamma^\ast$ exists and is unique; see \cite[Thm.~2.1]{branding2017} for a proof. We emphasize that, although $B_r$ is complete, for our purposes it suffices to roll along an arbitrarily short subarc of $\gamma$.
\end{remark}

\section{Proof of the main result}\label{sec:proof}
Here we prove Theorem~\ref{TH:main} and Corollary~\ref{COR:ConstSpeed}.

Let $h$ be the mean curvature of $S$ at $p$. We begin by showing that when $h\neq 0$ and $S$ is nonumbilic, the rolling of $B_{1/h}$ is well defined in a neighborhood of $p$.

\begin{lemma}
If $h\neq 0$ and $p$ is not an umbilic point, then for any unit vector $v \in T_pS$ and unit-speed curve $\gamma \colon [0, \ell] \to S$ such that $\gamma(0)=p$ and $\gamma'(0)=v$, there exists an $\varepsilon >0$ and a unique rolling of $B_{1/h}$ on $S$ along $\gamma\rvert_{[0,\varepsilon]}$.
\end{lemma}
\begin{proof}
It is clear that any curve on $B_{1/h}$ (oriented by $N_p$) has normal curvature $h$ and zero geodesic torsion. Hence, by Theorem~\ref{TH:rolling}, we need to show that if $\kappa_n(0) = h$ for some $v\in T_{p}S$, then $\tau_g(0) \neq 0$. This will imply that $\tau_g(t) \neq 0$ in a neighborhood of $0$, as desired.

Let $\kappa_1 \geq \kappa_2$ be the principal curvatures of $S$ at $p$, and let $e_1,e_2$ be a positively oriented basis of principal unit vectors. Given $v \in T_pS$, denote by $\theta$ the angle between $v$ and $e_1$. Then, by Euler's theorem,
\begin{align}
\kappa_n(\theta) &= \kappa_1 \cos^2\theta + \kappa_2 \sin^2\theta,\label{eq:kappaN}\\
\tau_g(\theta) &= (\kappa_2-\kappa_1)\sin\theta\cos\theta \label{eq:tauG}.
\end{align}

Suppose that $\kappa_n(\theta)= 1/2(\kappa_1+\kappa_2)$. Then $\theta = \pm\pi/4$ or $\theta = \pm3\pi/4$, which implies $\tau_g(\theta)= \pm(\kappa_1 - \kappa_2)/2$. Hence $\tau_g(\theta)$ can only vanish when $p$ is an umbilic point, and the lemma is proved.
\end{proof}

\begin{proof}[Proof of Theorem~\textup{\ref{TH:main}}]
By Theorem~\ref{TH:rolling}, we know that as $B_r$ rolls on $S$ along $\gamma$, the angular velocity at instant $t$ is given by
\begin{equation*}
\omega_t = \tau_g(t) \gamma'(t) + (1/r - \kappa_n(t) ) N(t)\times \gamma'(t).
\end{equation*}
The velocity of the center of $B_r$ is thus
\begin{equation*}
\omega_t \times rN(t) = (1-r\kappa_n(t))\gamma'(t) - r\tau_g(t) N(t)\times \gamma'(t),
\end{equation*}
which has length squared
\begin{equation}\label{eq:speed}
(1-r\kappa_n(t))^2 + r^2\tau_g^2(t).
\end{equation}

In particular, for $t=0$, we may substitute equations \eqref{eq:kappaN} and \eqref{eq:tauG} into \eqref{eq:speed}. This gives
\begin{equation*}
\mleft(1-r\mleft(\kappa_1 \cos^2\theta + \kappa_2 \sin^2\theta\mright)\mright)^2 + r^2\mleft((\kappa_2-\kappa_1)\sin\theta\cos\theta\mright)^2\eqqcolon \ell(\theta).
\end{equation*}
Expanding the squares and using the Pythagorean identity, we obtain
\begin{equation*}
\ell(\theta)=1+r^2\mleft( \kappa_2^2 -\kappa_2^2\cos^2\theta + \kappa_1^2\cos^2\theta\mright) -2r\mleft(\kappa_1 \cos^2\theta + \kappa_2 -\kappa_2\cos^2\theta\mright),
\end{equation*}
which we rewrite as
\begin{equation}\label{eq:speed2}
\ell(\theta)=1+r^2\kappa_2^2 -2r \kappa_2 + r(\kappa_2-\kappa_1)(2-r(\kappa_1+\kappa_2))\cos^2\theta.
\end{equation}
Since $\cos^2\theta =k$ has at most four solutions on $[0,2\pi]$ and $\ell(\theta)=\ell(\theta+\pi)$, equation~\eqref{eq:speed2} shows that if the initial speed is the same for at least three pairwise nonparallel directions, then it must be independent of $\theta$. Moreover, this occurs if and only if
\begin{equation*}
(\kappa_2-\kappa_1)(2-r(\kappa_1 + \kappa_2))=0,
\end{equation*}
that is, either $p$ is umbilic or $r=1/h$, which is the desired conclusion.
\end{proof}

\begin{proof}[Proof of Corollary~\textup{\ref{COR:ConstSpeed}}]
By equation~\eqref{eq:speed}, it is clear that planes and spheres are examples. A circular cylinder of radius $R$, oriented inward, has principal curvatures $1/R$ and $0$; taking $r=2R$, equation~\eqref{eq:speed2} gives $\ell(\theta)=1$, so circular cylinders are examples as well.

Conversely, suppose that the speed of $B_r$ is constant and equal to $c>0$ on $S$. If $S$ is totally umbilic, then it is part of a plane or a sphere. We may therefore assume that the set $\Omega$ of nonumbilic points is nonempty. At each $p\in\Omega$, choose an orthonormal principal basis, and write $\kappa_1>\kappa_2$ for the principal curvatures at $p$. By Theorem~\ref{TH:main},
\begin{equation}\label{eq:sum}
\kappa_1+\kappa_2=\frac{2}{r}.
\end{equation}
Substituting $r=2/(\kappa_1+\kappa_2)$ into \eqref{eq:speed2}, we obtain
\begin{equation*}
c^2 =1 + \frac{4\kappa_2^2}{(\kappa_1 + \kappa_2)^2} - \frac{4\kappa_2}{\kappa_1 + \kappa_2} =\frac{(\kappa_1-\kappa_2)^2}{(\kappa_1+\kappa_2)^2}.
\end{equation*}
Hence, on $\Omega$,
\begin{equation}\label{eq:difference}
(\kappa_1-\kappa_2)^2 = \frac{4c^2}{r^2}.
\end{equation}
Since the right-hand side is a positive constant and the principal curvatures depend continuously on $p$, every limit point of $\Omega$ is nonumbilic; hence $\Omega$ is closed. Since $\Omega$ is also open and $S$ is connected, we conclude that $\Omega=S$. It follows from \eqref{eq:sum} and \eqref{eq:difference} that both principal curvatures are constant. Therefore, $S$ is part of a plane, a sphere, or a circular cylinder. Since $S$ is nonumbilic, only the last possibility occurs.
\end{proof}

\section*{Acknowledgments}
The author thanks Steen Markvorsen and Vicente Palmer for interesting discussions.

\bibliographystyle{amsplain}
\bibliography{biblio}
\end{document}